\documentclass[12pt,a4paper]{ishiiart}                     
\usepackage{color}

\newtheorem{thm}{Theorem}[section]     
\newtheorem{lem}[thm]{Lemma}
\newtheorem{prop}[thm]{Proposition}
 
\theoremstyle{definition}

\theoremstyle{remark}

\newcommand{\tim}{\times}   
\newcommand{\R}{\mathbb R}
\newcommand{\pl}{\partial}
\newcommand{\AC}{{\rm AC}}
\newcommand{\cA}{\mathcal{A}}

\newcommand{\N}{\mathbb N}

\renewcommand{\d}{\,\mathrm{d}}

\newcommand{\SP}{\,\mathrm{SP}}
\newcommand{\disp}{\displaystyle}

\newcommand{\gth}{\theta}

\newcommand{\gep}{\varepsilon}
\newcommand{\gd}{\delta}
\newcommand{\gs}{\sigma}

\newcommand{\gb}{\beta}
\newcommand{\fr}{\frac}
\newcommand{\gO}{\Omega}
\newcommand{\hb}{\mbox}
\newcommand{\bye}{\end{document}}
\newcommand{\by}{\end{proof}\end{document}}

\def\ol#1{\,\overline{\!#1\!}\,}
\def\iol#1{\,\overline{\!#1}}

\def\go{\omega}
\def\gz{\zeta}

\def\Lip{\!\,{\rm Lip}\,}
\def\gg{\gamma}
\def\DT#1{\mathbf{\dot{#1}}}
\def\mid{\,:\,}
\def\aln{&\,}

\numberwithin{equation}{section}

\begin{document}

\title[Asymptotic Solutions]{Long-time asymptotic solutions  
of convex Hamilton-Jacobi equations with 
Neumann type boundary conditions}

\author{Hitoshi Ishii}
\address{Department of Mathematics, Waseda University, Nishi-Waseda, Shinjuku, 
Tokyo, 169-8050 Japan}
\email{hitoshi.ishii@waseda.jp}
\thanks{  
The author was supported in part by KAKENHI \#20340026, 
\#20340019 and \#21224001, JSPS}



\subjclass[2010]{Primary 35B40; Secondary 35F31, 35D40, 37J50, 49L25}   



\keywords{Asymptotic solutions, Hamilton-Jacobi equations, Neumann condition, 
Aubry sets, weak KAM theory}

\begin{abstract}
We study the long-time asymptotic behavior of solutions $u$ of the Hamilton-Jacobi 
equation $u_t(x,t)+H(x,\, Du(x,t))=0$ in $\gO\tim(0,\,\infty)$, where 
$\gO$ is a bounded open subset of $\R^n$,   
with Hamiltonian $H=H(x,p)$ being convex and coercive in $p$, and establish the uniform 
convergence of $u$ to an asymptotic solution as $t\to\infty$.      
\end{abstract}

\maketitle


\section{Introduction}
We study the long-time behavior of solutions of 
the initial-boundary (value) problem for 
the Hamilton-Jacobi equation
\begin{equation}\label{eq:1.1}
u_t(x,t)+H(x,Du(x,t))=0 \ \ \hb{ in }\gO\tim(0,\,\infty),
\end{equation} 
where $u=u(x,t)$ represents the unknown function on $\ol\gO\tim(0,\,\infty)$, 
$\gO$ is a bounded domain (i.e., open connected subset) of $\R^n$, $u_t:=\pl u/\pl t$, 
$Du:=(\pl u/\pl x_1,...,\pl u/\pl x_n)$ and $H=H(x,p)$ is the so-called Hamiltonian, 
which is a given continuous function on $\ol\gO\tim\R^n$ assumed to be 
convex in $p$.  
We are concerned with the Neumann type boundary condition 
\begin{equation}\label{eq:1.2}
D_\gg u(x,t)=g(x) \ \ \hb{ on }\pl\gO\tim(0,\,\infty), 
\end{equation}
where $g$ is a given continuous function on $\pl\gO$, 
$D_\gg u(x,t)$ denotes the 
derivative of $u$ in the direction of the 
vector $\gg(x)$, i.e., $D_\gg u(x,t):=\gg(x)\cdot Du(x,t)$, and 
$\gg$ is a given continuous vector field on $\pl\gO$ 
oblique to $\pl\gO$.   
The initial condition is given by a continuous function $u_0$ 
on $\ol\gO$. That is, 
\begin{equation}\label{eq:1.3}
u(x,0)=u_0(x) \ \ \hb{ for }x\in\ol\gO. 
\end{equation} 

In addition to the continuity 
of $H$, $g$, $\gg$ and $u_0$ and the boundedness of $\gO$,  
we make the following assumptions. 

\begin{description}
\setlength{\itemsep}{5pt}
\addtolength{\labelwidth}{20pt}
\item[(A1)] $H$ is a convex Hamiltonian, i.e., for each $x\in\ol\gO$ the function 
$H(x,\cdot)$ is convex on $\R^n$.
\item[(A2)] $H$ is coercive. That is, 
$\disp
\lim_{|p|\to\infty}H(x,p)=\infty
$ for all $x\in\ol\gO$. 
\item[(A3)] $\gO$ is a $C^1$ domain.  
\item[(A4)] $\gg$ is oblique to $\pl\gO$. That is, for any $x\in\pl\gO$, 
if $\nu(x)$ denotes the outer unit normal vector at $x$, then $\nu(x)\cdot \gg(x)>0$.  
\end{description}

According to \cite[Theorem 5.1]{Ishii_weakKAM_10} (see also \cite[Theorem 12]{Lions_Neumann_85}), 
we have the following existence and uniqueness  
theorem.   

\begin{thm}\label{thm:1.1} Under the above assumptions, there exists a unique solution   
$u\in C(\ol\gO\tim[0,\,\infty))$ of (\ref{eq:1.1})--(\ref{eq:1.3}). 
\end{thm}

The stationary problem associated with (\ref{eq:1.1})--(\ref{eq:1.3}) is the so-called 
ergodic problem or additive eigenvalue problem, that is, the problem of finding 
a pair of a constant $c\in \R$ and a function $w\in C(\ol\gO)$ such that 
$w$ is a solution of 
\begin{equation}\label{eq:1.4}
\left\{
\begin{aligned}  
&H(x,\,Dw(x))=c \ \ \hb{ in }\gO,\\
&D_\gg w(x)=g(x) \ \ \hb{ on }\pl\gO. 
\end{aligned}
\right.
\end{equation}

Existence results for the ergodic problem (\ref{eq:1.4}) go back to  
\cite[Section VII]{Lions_Neumann_85}. 
According to \cite[Theorem 6.1]{Ishii_weakKAM_10} or \cite[Theorem 14]{Lions_Neumann_85}, we have

\begin{prop}\label{prop:1.2}
There exists a unique constant $c\in\R$ such that (\ref{eq:1.4}) 
has a solution $w\in\Lip(\ol\gO)$. 
\end{prop}  

The unique constant $c$ given by the above proposition 
is called the \emph{critical value} and is given as 
the minimum value of $a\in\R$ for which 
problem (\ref{eq:1.4}), with $c$ replaced by $a$, has a subsolution.  
A general useful remark here is this: under condition (A2), 
every subsolution $w$ of (\ref{eq:1.4}), with any given constant $c$, is 
Lipschitz continuous on $\ol\gO$.

In the above two propositions, assumption (A1) is superfluous to obtain the 
stated conclusions, and we do not seek for optimal hypotheses for such conclusions.   
In fact, this work is a continuation of \cite{Ishii_weakKAM_10}, where 
the author \cite{Ishii_weakKAM_10} has studied (\ref{eq:1.4}) as well as (\ref{eq:1.1})--(\ref{eq:1.3}) 
in the view point of weak KAM theory under assumptions (A1)--(A4). 
  
The notion of solution of (\ref{eq:1.1}) and (\ref{eq:1.2}) or (\ref{eq:1.4}) 
adopted here is that of viscosity solution and we refer the reader to \cite{Barles-book_94, 
BardiCapuzzo_97, UG} for a general account 
of viscosity solutions theory. 

We set 
\begin{align*}
&Q=\{(x,p)\in\ol\gO\tim\R^n\mid H(x,p)=c\},\\
&S=\{(x,\xi)\in\ol\gO\tim\R^n\mid \hb{there exists }p\in\R^n\hb{ such that }(x,p)\in Q \hb{ and } 
\xi\in D_2^+H(x,p) \},
\end{align*}
where $c$ is the critical value.
  
We now introduce another assumption on $H$ and we assume in our main theorem 
that either of the following 
(A5)$_+$ or (A5)$_-$ holds:  

\begin{description}
\setlength{\itemsep}{5pt}
\addtolength{\labelwidth}{20pt}
\item[(A5)$_\pm$]There exists a modulus $\go$ satisfying $\go(r)>0$ for 
$r>0$ such that if $(x,\,p)\in Q$, $\xi\in D_2^+H(x,\,p)$ and $p'\in\R^n$, then
\[
H(x,\,p+p')\geq \xi\cdot(p+p')+\go((\xi\cdot p')_\pm).
\]
\end{description} 

In the above and in what follows, the term ``modulus'' is used to indicate a continuous, nondecreasing 
function $\go$ on $[0,\,\infty)$ such that $\go(0)=0$, and we use the notation: $r_+=\max\{r,,0\}$ 
and $r_-=\min\{r,\,0\}$ for $r\in\R$.     
 
We are now in position to state our main result. 

\begin{thm}\label{thm:1.3}  
Under the above hypotheses, there exists 
a solution $w\in\Lip(\ol\gO)$
of (\ref{eq:1.4}) such that   
$u(x,t)+ct$ converges to $w(x)$ uniformly for $x\in\ol\gO$ as $t\to \infty$.   
\end{thm}

The above theorem states that the solution $u(x,t)$ of (\ref{eq:1.1})--(\ref{eq:1.3}) 
``converges'' to 
an ``asymptotic solution'' $w(x)-ct$ of (\ref{eq:1.1}), (\ref{eq:1.2})    
uniformly on $\ol\gO$ as $t\to \infty$.  

The study of the long-time asymptotic behavior of solutions of Hamilton-Jacobi 
equations (\ref{eq:1.1}) has a long history, which goes back to \cite{Kruzkov_67, 
Lions-redbook_82, Barles_85}, 
and it has received an intense interest in recent years. For some recent developments, 
we refer to \cite{fathi_98, NamahRoquejoffre_99, BarlesSouganidis-largetime_00, 
Roquejoffre_01, DaviniSiconolfi_06, FujitaIshiiLoreti_06, Ishii_08,  
IchiharaIshii_semiperiodic_08, IchiharaIshii_1D_08, IchiharaIshii_09, 
Mitake-asymptotic_08, 
Mitake-large-time_08}. These literatures have established results similar to the above theorem 
in the case when $\gO$ is a compact manifold without boundary, typically 
an $n$-dimensional flat torus (\cite{fathi_98, NamahRoquejoffre_99, BarlesSouganidis-largetime_00, 
Roquejoffre_01, DaviniSiconolfi_06}), or in the case when $\gO$ is the whole 
$n$-dimensional Euclidean space under an appropriate behavior of solutions at infinity 
(\cite{BarlesRoquejoffre_ergodictype_06, FujitaIshiiLoreti_06, Ishii_08, 
IchiharaIshii_semiperiodic_08, IchiharaIshii_1D_08, IchiharaIshii_09}), 
or in the case of the state-constraints or the Dirichlet boundary conditions 
(\cite{Roquejoffre_01, Mitake-asymptotic_08, Mitake-large-time_08}).    
Concerning the Neumann boundary conditions, Theorem \ref{thm:1.3} 
is one of first, general results on the convergence of solutions of (\ref{eq:1.1}), 
(\ref{eq:1.2}) to asymptotic solutions. In this regard, the author recently 
learned that G. Barles and H. Mitake (\cite{BarlesMitake_10}) had obtained convergence results 
similar to the above theorem. 
They took a PDE approach similar to the one in \cite{BarlesSouganidis-largetime_00}, which is fairly  
different from ours. They do not assume the convexity of $H$ although our convergence result 
under (A5)$_-$ seems to be out of their scope.

Henceforth, by replacing $H$ by $H-c$ if necessary, we normalize that 
$c=0$. Thus, the problem 
\begin{equation}\label{eq:1.5}
\left\{
\begin{aligned}
&H(x,\,Dw(x))=0 \ \ \hb{ in }\gO,\\
&D_\gg w(x)=g(x) \ \ \hb{ on }\pl\gO. 
\end{aligned}
\right.
\end{equation}
has a solution. The conclusion of Theorem \ref{thm:1.3} is now stated as
the uniform convergence on $\ol\gO$ 
of the solution $u(x,t)$ of (\ref{eq:1.1})--(\ref{eq:1.3}) 
to a solution $w(x)$ of (\ref{eq:1.5}) as $t\to\infty$.     

In the next section we establishes a theorem which adapts \cite[Proposition 2.4]{Ishii_08} 
to accommodate the Neumann type boundary condition. 
In Section 3, we prove our main result, Theorem \ref{thm:1.3}. 
In Section 4, we give formulas for asymptotic solutions 
similar to \cite{DaviniSiconolfi_06, FujitaIshiiLoreti_06, Ishii_08, IchiharaIshii_semiperiodic_08, IchiharaIshii_09, Mitake-asymptotic_08, Mitake-large-time_08}, which are now standard observations.

\section{An existence result} 

We write $B_r$ for the open ball $\{x\in\R^n\mid |x|<r\}$, with $r>0$, 
and $e_n$ for the unit vector $(0,...,0,1)\in\R^n$. 
Let $I=[0,\,T]$, with $0<T<\infty$. 

In this section we will be devoted to proving the following theorem.

\begin{thm}\label{thm:2.1}Let $u\in C(\ol\gO)$ be a subsolution of (\ref{eq:1.5}). 
Let $\eta\in\AC(I,\,\R^n)$ be such that 
$\eta(t)\in\ol\gO$ for $t\in I$. Set $I_{\pl}=\{t\in I\mid \eta(t)\in\pl\gO\}$. 
Then there exists a function 
$p\in L^\infty(I,\,\R^n)$ such that 
\begin{align}
&\fr{\d }{\d t}u\circ\eta(t)=p(t)\cdot\DT{\eta}(t) \ \ \hb{ a.e. }t\in I, 
\label{eq:2.1}\\
&H(\eta(t),\,p(t))\leq 0 \ \ \hb{ a.e. }t\in I,
\label{eq:2.2}\\
&\gg(\eta(t))\cdot p(t)\leq g(\eta(t)) \ \ \hb{ a.e. }t\in I_\pl.
\label{eq:2.3} 
\end{align}
\end{thm}

\begin{lem}\label{lem:2.2}Let $U$ be an open subset of $\R^n$, $\phi\in\Lip(U)$ 
and $\eta\in\AC(I,\,\R^n)$.    
Assume that $\eta(t)\in U$ for all $t\in I$.  
Then there exists a function $q\in L^\infty(I,\,\R^n)$ such that
\begin{align*} 
&q(t)\in \pl_{{\rm C}}u(\eta(t)) \ \ \hb{ a.e. }t\in I,\\
&\fr{\d}{\d t}u\circ \eta(t)=q(t)\cdot \DT{\eta}(t)   
\ \ \hb{ a.e. }t\in I.
\end{align*}
Here $\pl_{{\rm C}}u(x)$ stands for the Clarke differential of $u$ at $x$, 
which is defined as 
\[
\pl_{{\rm C}}u(x)=\bigcap_{r>0}Du(x;r), 
\]
where $Du(x;r)$ denotes the closed convex hull of 
\[
\{Du(y)\mid y\in\R^n,\,|y-x|<r, u \hb{ is differentiable at }y\}. 
\]
\end{lem}


In what follows, given a vector $w\in\R^n$ we denote by the symbol $w^*$ 
the function given by $x\mapsto w\cdot x$. Typically we write  
$\{e_n^*\leq 0\}$ for the set 
$\{x\in\R^n\mid e_n\cdot x \leq 0\}$.  
Let $\phi$ be a function defined on a subset $V$ of $\R^n$. For 
$x\in V$, we denote 
by $D_V^+\phi(x)$ the superdifferential of $\phi$ at $x$, i.e., 
the set of points $p\in\R^n$ such that 
\[
\phi(y)\leq \phi(x)+p\cdot (y-x)+o(|y-x|) \ \ \hb{ as }y\in V \ \hb{ and } \ y\to x. 
\]  
When $V$ is a neighborhood of $x$, we write just $D^+\phi(x)$ for $D_V^+\phi(x)$.

\begin{lem}\label{lem:2.3}Let $r>0$ and $v\in\Lip(B_r\cap\{e_n^*\leq 0\})$. 
Let $L>0$ be a Lipschitz constant of $v$. 
Let $\gg_i\in\R^n$, with $i=1,2$, and $\gep>0$, and assume  
that $\gg_i\cdot e_n\geq 0$ for $i=1,2$ and $|\gg_1-\gg_2|\leq \gep/L$.  
Let $a\in\R$ and assume that 
\[ 
\min\{|q|-L,\, \gg_1\cdot q-a\}\leq 0 \ \ \hb{  
for all }q\in D^+_{B_r\cap\{e_n^*\leq 0\}}v(0).
\]
Let $p\in D^+_{B_r\cap\{e_n^*\leq 0\}}v(0)$ be such that $\gg_1\cdot p\leq a$. 
Then we have $\gg_2\cdot p\leq a+\gep$.
\end{lem}

\begin{proof} We set $P=D^+_{B_r\cap\{e_n^*\leq 0\}}v(0)$ and note that 
\begin{equation}\label{eq:2.4}
\min\{|q|-L,\, \gg_1\cdot q-a\}\leq 0 \ \ \hb{  for all }q\in \iol{P}.
\end{equation}
Note as well that $P$ and $\iol{P}$ are is convex. 

We may assume that $v$ is defined and Lipschitz continuous 
on $\iol{B}_r\cap\{e_n^*\leq 0\}$.  
Fix any $p\in D_{B_r\cap\{e_n^*\leq 0\}}^+v(0)$, and choose 
a function $\phi\in C^1(\iol{B}_r\cap\{e_n^*\leq 0\})$ such that $D\phi(0)=p$ 
and that $v-\phi$ attains a strict maximum at the origin. 
We may assume by adding a constant to $\phi$ that $(v-\phi)(0)=0$

For any $0<\gd\leq r$, we set 
\begin{align*}
&\Theta_\gd=\{t\in\R\mid \max_{\iol{B}_\gd\cap\{e_n^*\leq 0\}}(v-\phi-t e_n^*)>0 \}, \\
&t_\gd=\inf \Theta_\gd. 
\end{align*}
It is obvious that $(L+|p|,\,\infty)\subset \Theta_\gd \subset (0,\infty)$ and 
that if $0<\gd_1<\gd_2\leq r$, then $\Theta_{\gd_1}\subset \Theta_{\gd_2}$. Therefore, 
if $0<\gd_1<\gd_2\leq r$, then $L+|p|\geq t_{\gd_1}\geq t_{\gd_2}\geq 0$. 
We set 
\[
t_0=\lim_{\gd\to 0+}t_\gd, 
\]
and observe that 
\[
t_0 =\sup_{\gd>0}t_\gd\in\left[\,0,\,L+|p|\,\right].
\]

We note that if $\mu>0$, then  
\begin{equation}\label{eq:2.5}
\max_{\iol{B}_\gd\cap\{e_n^*\leq 0\}}(v-\phi-(t_0+\mu)e_n^*)>0 \ \ \hb{ for }0<\gd\leq r. 
\end{equation}
Fix any $\mu>0$, and for $\gb>0$ we set 
\[
\Phi_\gb(x)=v(x)-\phi(x)-(t_0+2\mu)e_n\cdot x-\gb(e_n\cdot x)^2 \ \ \hb{ for }x\in
\iol{B}_r\cap\{e_n^*\leq 0\}. 
\]
Let $x_\gb\in\iol{B}_r\cap\{e_n^*\leq 0\}$ be a maximum point of $\Phi_\gb$. 
Since $\Phi_\gb(x_\gb)\geq \Phi_\gb(0)=0$, we have 
\[
\gb(e_n\cdot x_\gb)^2\leq v(x_\gb)-\phi(x_\gb)-(t_0+2\mu) e_n\cdot x_\gb 
\]
from which we see that $x_\gb\cdot e_n \to 0$ as $\gb\to\infty$. 
Moreover, we deduce that  
\[
\liminf_{\gb\to\infty} (v-\phi)(x_\gb)\geq 0, 
\]
from which we conclude that $x_\gb\to 0$ as $\gb\to\infty$. 

Observe that if $-\mu/\gb<e_n\cdot x<0$, then 
\[
-\mu e_n\cdot x-\gb(e_n\cdot x)^2=|e_n\cdot x|(\mu-\gb|e_n\cdot x|)>0.
\]
From this and (\ref{eq:2.5}) we see that 
$\Phi_\gb(x_\gb)=\max_{\iol{B}_r\cap\{e_n^*\leq 0\}}\Phi_\gb >0$. 
In particular, we have 
$e_n\cdot x_\gb>0$. Thus, by choosing $\gb>0$ large enough, we may assume that 
$x_\gb\in B_r\cap\{e_n^*<0\}$, and consequently we have 
\[
0\in D^+\Phi_\gb(x_\gb)=D^+v(x_\gb)-D\phi(x_\gb)-(t_0+2\mu)e_n-2\gb(e_n\cdot x_\gb)e_n.
\]
Thus, if $\gb$ is large enough, then we have   
\begin{equation}\label{eq:2.6}
|D\phi(x_\gb)+(t_0+2\mu)e_n+2\gb(e_n\cdot x_\gb)e_n|\leq L.
\end{equation} 
 
Next, note that there is a constant $\rho=\rho_\mu\in(0,\,r)$ such that 
$t_0-\mu\leq t_\rho$, that is, 
\begin{equation}\label{eq:2.7}
\max_{\iol{B}_\rho\cap\{e_n^*\leq 0\}}(v-\phi-(t_0-\mu)e_n^*)=0. 
\end{equation}
By choosing $\gb$ large enough, we may assume that $x_\gb\in B_\rho$. Then we have 
\[
v(x_\gb)-\phi(x_\gb)-(t_0-\mu)e_n\cdot x_\gb 
\leq 0 <\Phi_\gb(x_\gb),
\]
and therefore 
\[
\gb(e_n\cdot x_\gb)^2<-3\mu e_n\cdot x_\gb=3\mu|e_n\cdot x_\gb|,
\] 
which yields 
\begin{equation}\label{eq:2.8}
\gb|e_n\cdot x_\gb|< 3\mu. 
\end{equation}

Due to (\ref{eq:2.7}), the function $\Psi:=v-\phi-(t_0-\mu)e_n^*$ attains a 
maximum at the origin over $B_\rho\cap\{e_n^*\leq 0\}$, and hence we have 
\[
0\in D^+_{B_\rho \cap\{e_n^*\leq 0\}}\Psi(0)
=P-D\phi(0)-(t_0-\mu)e_n. 
\] 
That is, we have
\begin{equation}\label{eq:2.9}
p+(t_0-\mu)e_n\in P.
\end{equation}

Sending $\gb\to\infty$ first and then $\mu\to 0$ and using (\ref{eq:2.8}),  
we obtain from (\ref{eq:2.6}) 
\[
|p+t_0 e_n|\leq L. 
\]
Also, from (\ref{eq:2.9}) we see that  
\[
p+t_0e_n \in\iol{P}.  
\]

We set $A=\{t\in[0,\,t_0]\mid |p+te_n|\leq L\}$ 
and $B=\{t\in[0,\,t_0]\mid \gg_1\cdot (p+te_n)\leq a\}$. 
By the convexity 
of $\iol{P}$, since $p,\,p+t_0e_n\in\iol{P}$, 
we see that $p+te_n\in\iol{P}$ for $t\in[0,\,t_0]$. In view of 
(\ref{eq:2.4}), we have $[0,\,t_0]=A\cup B$. Since 
$A$ and $B$ is closed sets, setting $\tau=\min B$, we have $\tau\in A\cap B$.   
Hence, we get 
\begin{align*}
\gg_2\cdot p 
\leq\aln  \gg_2\cdot (p+\tau e_n)
\leq \gg_1\cdot(p+\tau e_n)+|\gg_2-\gg_1||p+\tau e_n| \\
\leq\aln  a+L|\gg_2-\gg_1|\leq a+\gep,
\end{align*}
which completes the proof. 
\end{proof}

\begin{lem}\label{lem:2.4}Let $\gg\in C^\infty(\R^n)$ and $g\in C^\infty(\R^n)$ 
satisfy 
\[
\inf_{x\in\R^n}e_n\cdot\gg(x)>0 \  \ \hb{ and } \ \ 
\sup_{x\in\R^n}(|\gg(x)|+|g(x)|)<\infty. 
\]
Then there exists a function $\psi\in C^\infty(\R^n)$ such that 
\[
\gg(x)\cdot D\psi(x)=g(x) \ \ \hb{ for }x\in\R^n, 
\]
\end{lem}
  
The above assertion is well-known, but 
for completeness we give a proof.  
  
\begin{proof} The idea is to solve the initial value problem for the linear PDE
\begin{equation} \label{eq:2.10}\left\{
\begin{aligned}
&\gg(x)\cdot D\psi(x) =g(x)  \ \ \hb{ in }\R^n, \\
&\psi(x)=0 \ \ \ \hb{ if }e_n\cdot x=0. 
\end{aligned}
\right.\end{equation}
For $(x,t)\in\R^n\tim\R$ let $\Phi(x,t)$ denote the (unique) solution of 
the initial value problem for ODE
\[\left\{
\begin{aligned}
&\Phi_t(x,t)=\gg(\Phi(x,t)) \ \ \hb{ for }t\in \R,\\
&\Phi(x,0)=x,
\end{aligned}
\right.
\]
where $\Phi_t:=\pl\Phi/\pl t$. By the standard ODE theory we see that $\Phi\in C^\infty(\R^{n+1})$. 
Moreover, since 
\[
0<\inf_{\R^n}e_n\cdot\gg\leq 
e_n\cdot \gg(\Phi(x,t))=e_n\cdot \Phi_t(x,t)\leq \sup_{\R^n}|\gg|<\infty, 
\]
we see that for each $x\in\R^n$, there exists a unique $\tau(x)\in\R$ such that 
$e_n\cdot \Phi(x,\,\tau(x))=0$. Then the implicit function theorem
guarantees that $\tau\in C^\infty(\R^n)$.

We define $\psi\in C^\infty(\R^n)$ by setting 
\[
\psi(x)=-\int_0^{\tau(x)}g(\Phi(x,t))\d t.  
\]
It is obvious that $\psi(x)=0$ if $e_n\cdot x=0$. 
For $r\in\R$ we have
\begin{align*}
\psi(x)=\aln -\int_0^r g(\Phi(x,t))\d t -\int_r^{\tau(x)}g(\Phi(x,t))\d t\\
=\aln -\int_0^r g(\Phi(x,t))\d t-\int_0^{\tau(x)-r}g(\Phi(\Phi(x,\,r),\,t))\d t\\
=\aln-\int_0^r g(\Phi(x,t))\d t+\psi(\Phi(x,\,r)). 
\end{align*}
Differentiating the above by $r$ and setting $r=0$, we get 
\[
0=-g(x)+\gg(x)\cdot D\psi(x). 
\]
Thus the function $\psi$ is a solution of (\ref{eq:2.10}), which completes the proof.  
\end{proof}

In the next lemma, we assume that the vector field $\gg$ is of class $C^1$.

\begin{lem}\label{lem:2.5}Let $r>0$, 
$G\in C(\iol{B}_r\tim\R^n,\,\R)$ and $\gg\in C^1(\iol{B}_r,\,\R^n)$. 
Assume that $G$ satisfies (A2), with $\gO$ replaced by $B_r$, and that 
$e_n\cdot \gg(x)>0$ for all $x\in\iol{B}_r$. 
Let $v\in C(\iol{B}_r\cap\{e_n^*\leq 0\})$ and $\gep>0$, and assume 
that $v$ is a subsolution of
\begin{equation}\label{eq:2.11}
\left\{
\begin{aligned}
&G(x,\,Dv)=0 \ \ \hb{ in }B_r\cap\{e_n^*<0\},\\
&D_\gg v(x)=-\gep \  \ \hb{ on }B_r\cap\{e_n^*=0\}.  
\end{aligned}
\right.
\end{equation}
Then there exists 
a function $w\in \Lip(B_{r/2}\cap\{e_n^*<\gd\})$, with $\gd>0$,  such that 
\[
|v(x)-w(x)|<\gep \ \ \hb{ for }x\in B_{r/2}\cap\{e_n^*\leq 0\}, 
\]
and $\phi$ is both a 
subsolution of 
\[
G(x,\,Dw(x))=\gep \ \ \hb{ in }B_{r/2}\cap\{e_n^*<\gd\}
\] 
and of 
\[
\gg(x)\cdot Dw(x)=\gep \ \ \hb{ in }B_{r/2}\cap\{|e_n^*|<\gd\}.
\]
\end{lem}

We remark that, by definition, $v$ is a subsolution of (\ref{eq:2.11}) if and only if 
\[
\left\{
\begin{aligned}
&G(x,p)\leq 0 \ \ \hb{ for }p\in D^+ v(x) \ \hb{ and } \ x\in B_r\cap\{e_n^*<0\}, \\
&G(x,p)\wedge \left(\gg(x)\cdot p+\gep\right)\leq 0\ \ \hb{ for }p\in D^+_{B_r\cap\{e_n^*\leq 0\}}v(x) 
\ \hb{ and } \ x\in B_r\cap\{e_n^*=0\}.   
\end{aligned}\right.
\]

In order to prove the above lemma we need the following lemma.

\begin{lem}\label{lem:2.6}There exists a function 
$\gz\in C^\infty(\R^n_+\tim\R^n)$, where $\R^n_+:=\R^{n-1}\tim (0,\,\infty)=\{e_n^*>0\}$, such that
\[\left\{
\begin{aligned}
&\gz(\xi,tz)= t^2\gz(\xi,z) \ \ &&\hb{ for }(\xi,z,t)\in\R^n_+\tim\R^n\tim\R,\\
&\gz(\xi,z)> 0 \ \ &&\hb{ for }(\xi,z)\in\R^n_+\tim(\R^n\setminus\{0\}),\\
&\xi\cdot D_z\gz(\xi,z)=(e_n\cdot \xi)(e_n\cdot z) \ \ &&\hb{ for }(\xi,z)\in\R^n_+\tim\R^n.   
\end{aligned}
\right.
\]  
\end{lem}

We refer the reader to \cite[Lemma 2.3]{Ishii_weakKAM_10} for a proof of the above lemma. 
For the proof of Lemma \ref{lem:2.5}, we follow that of \cite[Lemma 2.5]{Ishii_weakKAM_10}. 
   
\begin{proof}[Proof of Lemma \ref{lem:2.5}] In view of (A2) we may 
choose a constant $L>0$ so that 
\[
\{G(x,\,\cdot)\leq 0\}\subset \iol{B}_L \ \ \hb{ for }x\in \iol{B}_r.
\] 
It is easily seen that $v$ is Lipschitz continuous on $\iol{B}_r\cap 
\{e_n^*\leq 0\}$, with $L$ as its Lipschitz constant.  Let $\go_G$ be the modulus 
of continuity of the function $G$ on $B_r\tim B_L$.    

Let $\gz\in C^\infty(\R_+^n\tim\R^n)$ be the function given 
by Lemma \ref{lem:2.6}. 
We note by the homogeneity of the functions $\gz(\xi,\,\cdot)$ that 
\[
C_0^{-1}|z|^2\leq \gz(\xi,z)\leq C_0|z|^2, \qquad |D_\xi\gz(\xi,z)|  
\leq C_0|z|^2,  \qquad  
|D_z\gz(\xi,z)|\leq C_0|z|
\]
for all $(\xi,z)\in\R^n_+\tim\R^n$ and for some constant $1<C_0<\infty$.  
We set $\psi(x,\,y)=\gz(\gg(x),\,x-y)$ and note that 
\begin{align*}
D_x\psi(x,\,y)=&\,(D\gg(x))^{{\rm T}}D_\xi\gz(\gg(x),\,x-y)
+D_z\gz(\gg(x),x-y),\\
D_y\psi(x,\,y)=&\,-D_z\gz(\gg(x),\,x-y),
\end{align*}
where $A^{\rm T}$ denotes the transposed matrix of the matrix $A$. 
From these we get
\begin{equation}\label{eq:2.12}
|D_x\psi(x,y)+D_y\psi(x,y)|=\left|(D\gg(x))^{{\rm T}}D_\xi\gz(\gg(x),\,x-y)\right|
\leq C_0C_1|x-y|^2,  
\end{equation}
where $C_1>0$ is a bound of $|D\gg(x)|$ over $x\in \iol{B}_r$.

For $0<\gd<1$  
we define the sup-convolution $v^\gd\in C(\iol{B}_r)$ 
by
\[
v^\gd(x)=\max_{y\in \iol{B}_r\cap \{e_n^*\leq 0\}}
\left(v(y)-\fr{1}{\gd}\psi(x,\,y)\right).
\]
It is well-known and easy to see that $v^\gd(x)\to v(x)$ uniformly on 
$\iol{B}_r\cap\{e_n^*\leq 0\}$ as $\gd \to 0$ and that $v^\gd$ is Lipschitz 
continuous on $\iol{B}_r$.   

Henceforth we fix any $x\in B_{r/2}\cap\{e_n^*<\gd^2\}$, and    
choose a maximizer $y\in\iol{B}_r\cap\{e_n^*\leq 0\}$ 
of the above formula, so that we have  
\[
v^\gd(x)=v(y)-\fr{1}{\gd}\psi(x,\,y).   
\]
  
We collect here some estimates based on this choice of $x,\,y$.  
Let $\hat x$ denote the projection of $x$ onto the half space $\{e_n^*\leq 0\}$.
That is, $\hat x=x-(e_n\cdot x)e_n$ if $e_n\cdot x>0$ 
and $\hat x=x$ otherwise. 
Noting that $\hat x\in B_{r/2}\cap\{e_n^*\leq 0\}$ 
and $|x-\hat x|<\gd$, we get   
\[v^\gd(x)\geq v(\hat x)-\fr{1}{\gd}\psi(x,\,\hat x), 
\] 
and moreover 
\begin{align*}
\fr{1}{\gd}\psi(x,y)=\aln v(y)-v^\gd(x)
\leq v(y)-v(\hat x)+\fr 1\gd \psi(x,\,\hat x)\\       
\leq\aln L|y-\hat x|+\fr {C_0|\hat x-y|^2}\gd   
\leq L|x-y|+(L+C_0)\gd. 
\end{align*}
Since $\psi(x,y)\geq C_0^{-1}|x-y|^2$, we get
\[
|x-y|^2\leq C_0L\gd|x-y|+C_0(L+C_0)\gd^2,
\]
from which we deduce that   
\begin{equation}\label{eq:2.13}
|x-y|\leq C_2\gd,  
\end{equation}  
where \[
C_2=\fr{C_0L+\sqrt{(C_0L)^2+4C_0(C_0+L)}}{2}.  
\]    
We note from (\ref{eq:2.12}) that 
\begin{equation}\label{eq:2.14}
|D_x\psi(x,y)+D_y\psi(x,y)|\leq C_3\gd^2, 
\end{equation} 
where $C_3:=C_0C_1C_2^2$. By Lemma \ref{lem:2.6}, we get
\begin{equation} \label{eq:2.15} 
\gg(x)\cdot D_y\psi(x,y)=-
e_n\!\cdot\! \gg(x)\,e_n\!\cdot\!(x-y).
\end{equation}    
Also, we get
\begin{equation}  
|D_y\psi(x,y)|\leq C_0|x-y|\leq C_4\gd, \label{eq:2.16}
\end{equation}
where $C_4:=C_0C_2$. 
 
With $x$ and $y$ fixed as above, we show 
that if $\gd>0$ is sufficiently small, then 
\begin{equation}\label{eq:2.17}
G(x,\,p)\leq \gep \ \ \hb{ for }p\in D^+v^\gd(x). 
\end{equation}
We choose a constant $\gd_1\in(0,\,1)$ so that $C_2\gd_1<r/2$ and assume in what follows that 
$0<\gd<\gd_1$. By (\ref{eq:2.13}), we have $|x-y|<r/2$. Hence, we have 
$y\in B_r$. 

Fix any $p\in D^+v(x)$ and choose a function $\phi\in C^1(B_r\cap\{e_n\leq 0\})$ 
so that $v-\phi$ attains a maximum at $x$. 

We separate the argument into two cases.  
We first argue the case when $e_n\cdot y<0$. By a simple calculus, we have
\[
\fr{1}{\gd}D_y\psi(x,y)\in D^+v(y) \ \hb{ and } \ 
\fr 1\gd D_x\psi(x,y)+D\phi(x)=0. 
\] 
Hence, by assumption, we have 
\begin{align*}
0\geq \aln G\Big(y,\,\fr 1\gd D_y\psi(x,y)\Big)
\geq G(x,\,p)-\go_G(|x-y|)-\go_G\Big(\fr 1\gd|D_x\psi(x,y)+D_y\psi(x,y)|\Big)\\
\geq\aln G(x,\,p)-\go_G(C_2\gd)-\go_G(C_3\gd).  
\end{align*}
We choose a constant $\gd_2\in(0,1)$ so that $\go_G(C_2\gd_2)+\go_G(C_3\gd_2)<\gep$. 
Then, assuming that $0<\gd<\gd_1\wedge \gd_2$, we have $G(x,\,p)\leq \gep$. 

Next, we turn to the case where $e_n\cdot y=0$. Then we have 
\[
D\phi(x)=-\fr 1\gd D_x\psi(x,y)\in D^+v^\gd(x) \ \ \hb{ and } \ \ 
\fr 1\gd D_y\psi(x,y)\in D^+_{B_r\cap\{e_n^*\leq 0\}}v(y).
\]
Using (\ref{eq:2.15}), we compute  
\[
\gg(x)\cdot D_y\psi(x,y)
=-e_n\!\cdot\!\gg(x)\ e_n\!\cdot\! (x-y)
=-e_n\!\cdot\!\gg(x)\ e_n\!\cdot\! x
> -C_5\gd^2, 
\] 
where $C_5>0$ is a bound of $\sup_{B_r}|\gg|$.  
Since $|D_y\psi(x,y)|/\gd\leq C_4$ by (\ref{eq:2.16}),  
we get
\begin{align*}
\gg(y)\cdot \fr 1\gd D_y\psi(x,y)
=&\, \gg(x)\cdot\fr 1\gd D_y\psi(x,y)+\left(\gg(y)-\gg(x)\right)
\cdot\fr 1\gd D_y\psi(x,y)\cr
>&\,-C_5\gd-C_4\go_\gg(|x-y|)\geq -C_5\gd-C_4\go_\gg(C_2\gd), 
\end{align*}
where $\go_\gg$ denotes the modulus of continuity of $\gg$. 
We select a $\gd_3>0$ so that $C_4\go_\gg(C_2\gd_3)+C_5\gd_3<\gep$, and 
assume that $0<\gd < \gd_1\wedge \gd_3$. 
Then we have $\gg(y)\cdot D_y\psi^\gd(x,y)/\gd>-\gep$. 
Since $v$ is a viscosity subsolution of (\ref{eq:2.11}),  
we get $G\left(y,\,D_y\psi(x,y)/\gd\right)\leq 0$. 
Now, as in the previous case, we obtain 
\[
0\geq G\left(x,\,D\phi(x)\right)-\go_G(C_2\gd)-\go_G(C_3\gd). 
\]
Consequently, if $0<\gd<\gd_1\wedge \gd_2\wedge \gd_3$, then we have 
$G(x,\,p)\leq \gep$. That is, in both cases, 
inequality (\ref{eq:2.17}) holds if $0<\gd<\gd_1\wedge\gd_2\wedge\gd_3$.

Now, we assume in addition that $|e_n\cdot x|<\gd^2$ and 
show that if $\gd>0$ is sufficiently small, then 
\begin{equation}\label{eq:2.18}
\gg(x)\cdot p\leq\gep \ \ \hb{ for }p\in D^+v^\gd(x). 
\end{equation}
We fix any $p\in D^+v(x)$ and choose a function $\phi\in C^1(\iol{B}_r)$ 
so that $v^\gd-\phi$ attains a maximum at $x$ and $D\phi(x)=p$. 
For sufficiently small $t>0$, we have $y-t\gg(x)\in B_r$ and hence
\[
v^\gd(x-t\gg(x))\geq v(y)-\fr 1\gd\psi(x-t\gg(x),\,y). 
\]
Thus, for sufficiently small $t>0$, we have 
\[
\phi(x)-\phi(x-t\gg(x))\leq 
v^\gd(x)-
v^\gd(x-t\gg(x))\leq -\fr{1}{\gd}\left(\psi(x,y)-\psi(x-t\gg(x),y)\right),
\]    
which readily yields 
\[
\gg(x)\cdot D\phi(x)\leq -\gg(x)\cdot \fr 1\gd D_x\psi(x,y). 
\]
Noting that $e_n\cdot y\leq 0$ and $|e_n\cdot x|<\gd^2$, we observe by (\ref{eq:2.15}) 
that 
\[ 
\gg(x)\cdot D_y\psi(x,y)\leq -e_n\!\cdot\!\gg(x)\ e_n\!\cdot\! x
<C_5\gd^2.
\]
Using these observations together with (\ref{eq:2.14}), we obtain 
\[
\begin{aligned}
\gg(x)\cdot p\leq\aln  -\gg(x)\cdot \fr 1\gd D_x\psi(x,y) \\
\leq\aln \gg(x)\cdot\fr{1}{\gd}D_y\psi(x,y)+\fr{|\gg(x)|}\gd |D_x\psi(x,y)+D_y\psi(x,y)|
\leq (C_5 +C_3C_5)\gd. 
\end{aligned}
\]
Choosing a constant $\gd_4>0$ so that $C_5(1+C_3)\gd_4<\gep$, we find that 
if $0<\gd<\gd_4$, then (\ref{eq:2.18}) holds. 

Finally, we may choose a constant $\gd_5>0$ so that $|v(x)-v^\gd(x)|<\gep$ for all 
$x\in B_{r/2}\cap\{e_n^*\leq 0\}$ and $0<\gd<\gd_5$.  
Fixing a constant $0<\gd<\min_{1\leq i \leq 5} \gd_i$ 
and setting $w=v^\gd$,  
we see that $w$ satisfies the required properties with $\gd^2$ in place of $\gd$.  
\end{proof}

\begin{proof}[Proof of Theorem \ref{thm:2.1}]
It is enough to show that for each $\tau\in I$ there is a function 
$p_\tau\in L^\infty(I_\tau,\,\R^n)$ for some constant $\gd=\gd_\tau>0$, where $I_\tau
:=I\cap [\tau-\gd,\,\tau+\gd]$,  
such that (\ref{eq:2.1}), (\ref{eq:2.2}) and (\ref{eq:2.3}) hold with $I_\tau$ 
in place of $I$. 

We fix any $\tau\in I$. If $\eta(\tau)\in\gO$, then there is a constant $\gd>0$ 
such that for $t\in I_\tau:=I\cap [\tau-\gd,\,\tau+\gd]$, we have $\eta(t)\in\gO$. 
Lemma \ref{lem:2.2} then guarantees that there is a function $p_\tau\in L^\infty(I_\tau,\,\R^n)$ 
such that (\ref{eq:2.1})--(\ref{eq:2.3}) hold with $I_\tau$ and $p_\tau$    
in place of $I$ and $p$, respectively. 

We may therefore assume that $\eta(\tau)\in\pl\gO$.  
By making a $C^1$ change of variables, we may assume 
that $\eta(\tau)=0$ and that there is a 
constant $r>0$ such that 
$B_r\cap \gO\subset \{e^*_n<0\}$ and 
$\iol{B}_r\cap \{e^*_n<0\}\subset \gO$.   

We choose a constant $\gd>0$ so that $\eta(t)\in B_{r/2}$ for all $t\in I_\tau:=[\tau-\gd,\, 
\tau+\gd]\cap I$. 
We choose  
a constant $L>0$ so that $\{H(x,\,\cdot)\leq 0\}\subset \iol{B}_L$ for $x\in \gO$. 
An immediate consequence is that $u$ is Lipschitz continuous 
on $\ol\gO$ with $L$ 
as its Lipschitz constant.

We may assume that $\gg$ and $g$ are defined on $\ol \gO$ 
as continuous functions. We fix any $\gep>0$ and choose 
functions $\gg_\gep\in C^\infty(\iol{B}_r,\,\R^n)$  
and $g_\gep\in C^\infty(\iol{B}_r,\,\R)$ so that 
\[
|\gg_\gep(x)-\gg(x)|\leq \fr{\gep}L \ \ \hb{ and } \ \ |g_\gep(x)-g(x)|\leq \gep 
\ \ \hb{ for all }x\in B_r\cap \{e_n^*\leq 0\}.
\]
We may assume furthermore by replacing $r>0$ by a smaller one if needed   
that $\gg_\gep(x)\cdot e_n>0$ for all $x\in \iol{B}_r$.

We remark here that $u$ is a subsolution of 
\begin{equation}\label{eq:2.19}
\left\{
\begin{aligned}
&H(x,\,Du(x))\leq 0 \ \ \hb{ in }B_r\cap\{e_n^*<0\}, \\
&D_{\gg_\gep} u(x)=g_\gep(x)+2\gep \ \ \hb{ on }B_r\cap\{e_n^*=0\}. 
\end{aligned}
\right.
\end{equation}
To see this, let $x\in B_r\cap\{e_n^*=0\}$ and $p\in D_{B_r\cap\{e_n^*\leq 0\}}^+u(x)$. 
We have two cases, either $H(x,\,p)\leq 0$ or $\gg(x)\cdot p\leq g(x)$. 
If $H(x,\,p)\leq 0$, then we are done. Otherwise,  
applying Lemma \ref{lem:2.3}, with $B_r$ replaced by a small ball centered at $x$, we find that
\[
\gg_\gep(x)\cdot p\leq g(x)+\gep\leq g_\gep(x)+2\gep.
\] 
That is, $u$ is a subsolution of (\ref{eq:2.19}).

Thanks to Lemma \ref{lem:2.4}, there is a function $\psi_\gep\in C^\infty(\iol{B}_r)$ such that 
\begin{equation}\label{eq:2.20} 
\gg_\gep(x)\cdot D\psi_\gep(x)=g_\gep(x)+3\gep \ \ \hb{ for }x\in B_r. 
\end{equation} 
We have used here the fact that $\gg_\gep$ and $g_\gep$ can be extended to $C^\infty$ functions 
on $\R^n$ so that $\inf_{x\in\R^n}\gg_\gep(x)\cdot e_n>0$ and 
$\sup_{\R^n}(|g_\gep|+|\gg_\gep|)<\infty$. 

We may assume by extending $H$ to $\iol{B}_r\cap \{e_n^*>0\}$ in an appropriate manner 
that $H$ is defined and continuous at least 
on $\iol{B}_r\tim\R^n$ and satisfies (A2), with $B_r$ in place of $\gO$.  
We set $v_\gep:=u-\psi_\gep$ on $\iol{B}_r\cap\{e_n^*\leq 0\}$ and 
$G_\gep(x,\,p)=H(x,\,p+D\psi_\gep(x))$ for $(x,p)\in (\iol{B}_r\tim\R^n$. 
It is obvious that $v_\gep$ is a subsolution of $G_\gep(x,\,Dv_\gep(x))=0$
in $B_r\cap\{e_n^*<0\}$. Moreover, it is easily checked that  $v_\gep$ is a subsolution of 
\[
\left\{
\begin{aligned}
&G_\gep(x,\,Dv_\gep(x))\leq 0 \ \ \hb{ in }B_r\cap\{e_n^*<0\}, \\
&D_{\gg_\gep} v_\gep(x)=-\gep \ \ \hb{ on }B_r\cap\{e_n^*=0\}. 
\end{aligned}
\right.
\]

According to Lemma \ref{lem:2.5}, there exists a function 
$\phi_\gep\in C(\iol{B}_{r/2}\cap\{e_n^*\leq \rho_\gep\})$, with $\rho_\gep>0$, such that 
\[
|v_\gep(x)-\phi_\gep(x)|<\gep \ \ \hb{ for }x\in B_r\cap\{e_n^*\leq 0\}, 
\]
and $\phi_\gep$ is both a subsolution of 
\begin{equation}\label{eq:2.21}
G_\gep(x,\,D\phi_\gep(x))=\gep \ \ \hb{ in }B_{r/2}\cap\{e_n^*<\rho_\gep\}, 
\end{equation}
and of 
\begin{equation}\label{eq:2.22}
\gg_\gep(x)\cdot D\phi_\gep(x)=\gep \ \ \hb{ on }B_{r/2}\cap\{|e_n^*|<\rho_\gep\}. 
\end{equation}

Now, according to Lemma \ref{lem:2.2}, there is a function 
$q_{\gep}\in L^\infty(I_\tau,\,\R^n)$ such that $q_{\gep}(t)\in\pl_{{\rm C}}\phi_\gep(\eta(t))$ 
and $(\d\phi_\gep\circ \eta/\d t)(t)=q_{\gep}(t)\cdot \DT\eta(t)$ for a.e. $t\in I_\tau$, 
The last equality can be stated as 
\begin{equation}\label{eq:2.23}
\phi_\gep(\eta(t))-\phi_\gep(\eta(\tau))=\int_\tau^t q_{\gep}(s)\cdot \DT\eta(s)\d s 
\ \ \hb{ for }t\in I_\tau. 
\end{equation}
Setting 
\[I_{\tau,\pl}=\{t\in I_\tau\mid e_n\cdot \eta(t)=0\}.
\]
and noting that $\eta(t)\in B_{r/2}\cap\{e_n^*\leq 0\}$ for $t\in I_\tau$, 
from (\ref{eq:2.21}) and (\ref{eq:2.22})
we get
\begin{align}
&G_\gep(\eta(t),\,q_{\gep}(t))\leq \gep \ \ \hb{ for a.e. }t\in I_\tau, 
\label{eq:2.24}\\
&\gg_\gep(\eta(t))\cdot q_{\gep}(t)\leq \gep \ \ \hb{ for a.e. }t\in
I_{\tau,\pl}.\label{eq:2.25}
\end{align}
  
We set $p_{\gep}(t)=q_{\gep}(t)+D\psi_\gep(\eta(t))$ for 
$t\in I_\tau$. Then (\ref{eq:2.23}), (\ref{eq:2.24}) and (\ref{eq:2.25}) read 
\begin{align}
&\phi_\gep(\eta(t))-\phi_\gep(\eta(\tau))=\int_\tau^t (p_{\gep}(s)-D\psi_\gep(s))\cdot \DT\eta(s)\d s 
\ \ \hb{ for }t\in I_\tau. \label{eq:2.26}\\
&H(\eta(t),\,p_{\gep}(t))\leq \gep \ \ \hb{ for a.e. }t\in I_{\tau}, \label{eq:2.27}\\
&\gg_\gep(\eta(t))\cdot\big(p_{\gep}(t)-D\psi_\gep\left(\eta(t)\right)\big)
\leq \gep \ \ \hb{ for a.e. }t\in I_{\tau,\pl}. \label{eq:2.28}
\end{align}
Combining this (\ref{eq:2.28}) and (\ref{eq:2.20}), we get 
\[
\gg_\gep(\eta(t))\cdot p_{\gep}(t)\leq g_\gep(\eta(t))+4\gep \ \ \hb{ for a.e. }
t\in I_{\tau,\pl}. 
\] 
By (\ref{eq:2.26}), we have
\[
(\phi_\gep+\psi_\gep)(t)=(\phi_\gep+\psi_\gep)(\tau)
+\int_\tau^t p_\gep(s)\cdot \DT\eta(s)\d s \ \ 
\hb{ for }t\in I_\tau. 
\]
We note here that $\phi_\gep+\psi_\gep\to u$ uniformly on $\iol{B}_r\cap\{e_n^*\leq 0\}$ 
as $\gep\to 0$.  
From (\ref{eq:2.27}), we find that 
\[
|p_{\gep}(t)|\leq L \ \ \hb{ for a.e. }t\in I_{\tau}. 
\]
Hence, there is a sequence $\gep_j \to 0+$ such that 
the sequence $\{p_{\gep_j}\}_{j\in\N}$ converges to some function 
$p$ on $I_\tau$ weakly-star in $L^\infty(I_\tau,\,\R^n)$. It is a standard 
observation that there is a sequence $\{\pi_{j}\}_{j\in\N}
\subset L^\infty(I_\tau,\,\R^n)$ such that 
$\pi_{j}(t)$ converges to $p(t)$ for a e. $t\in I_\tau$ 
and for each $j$, the function $\pi_{j}$ is a convex combination 
of $\{p_{\gep_k}\}_{k\geq j}$. Thus, sending $j\to\infty$,  we find that 
\begin{align*}
u(\eta(t))-u(\eta(\tau))=\aln \int_\tau^t p(s)\cdot \DT\eta(s)\d s  
\ \ \hb{ for }t\in I_\tau, \\
H(\eta(t),\,p(t))\leq \aln 0 \ \  \hb{ for a.e. }t\in I_\tau, \\
\gg(\eta(t))\cdot p(t)\leq \aln g(\eta(t)) 
\ \ \hb{ for a.e. }t\in I_{\tau,\pl}.
\end{align*}
The proof is now complete. 
\end{proof}

\section{Proof of Theorem \ref{thm:1.3}}

We start by recalling some results established in \cite{Ishii_weakKAM_10} and needed in this section. 

We write $J=[0,\,\infty)$ for simplicity of notation. 
The Skorokhod problem associated with $(\gO,\,\gg)$ is to find a pair $(\eta,l)\in \Lip(J,\,\R^n) 
\tim L^\infty(J,\,R)$, for given $x\in\ol\gO$ and $v\in L^\infty(J,\,\R^n)$, 
such that 
\begin{equation}\label{eq:3.1}
\left\{
\begin{aligned}
&\eta(0)=x,\qquad \eta(s)\in\ol\gO \ \ \hb{ for all }s\in J,\\
&l(s)\geq 0 \ \ \hb{ for a.e. }s\in J,\\ 
&l(s)=0 \ \ \hb{ for a.e. } s\in J \ \hb{ such that }\eta(s)\in\gO,\\
&\DT\eta(s)+l(s)\gg(\eta(s))=v(s) \ \ \hb{ for a.e. }s\in J.  
\end{aligned}
\right.
\end{equation} 
According to  \cite[Theorem 4.2]{Ishii_weakKAM_10}, problem (\ref{eq:3.1}) has a solution. 
For $x\in\ol\gO$, we denote by $\SP(x)$ the set of all triples 
\[
(\eta,v,l)\in\Lip(J,\,\R^n)\tim L^\infty(J,\R^n) 
\tim L^\infty(J,\,\R)
\]
which satisfy (\ref{eq:3.1}). 

Let $L$ denote the Lagrangian of $H$. That is, $L(x,\,\xi)
=\sup_{p\in\R^n}\{\xi\cdot p-L(x,\,p)\}$ for $(x,\,\xi)\in\ol\gO\tim\R^n$. 
Thanks to \cite[Theorem 5.1]{Ishii_weakKAM_10}, we know that if $u$ is a solution of (\ref{eq:1.1})--(\ref{eq:1.3}), 
then 
\[
\begin{aligned}
u(x,t)=\inf\Big\{
&\int_0^t \big(L(\eta(s),\,-v(s))+g(\eta(s))l(s)\big)\d s\\ 
&+u_0(\eta(t))\mid (\eta,v,l)\in\SP(x)\Big\} \ \ 
\hb{ for }(x,t)\in\ol\gO\tim(0,\,\infty). 
\end{aligned}
\]
The dynamic programming principle yields for $(x,t)\in\ol\gO\tim(0,\,\infty)$ 
and $0<\tau<t$,  
\[
\begin{aligned}
u(x,t)=\inf\Big\{
&\int_0^\tau \big(L(\eta(s),\,-v(s))+g(\eta(s))l(s)\big)\d s\\
&+u(\eta(\tau),\,t-\tau)\mid (\eta,v,l)\in\SP(x)\Big\}.
\end{aligned}
\]

We assume throughout this section that $c=0$, i.e., problem (\ref{eq:1.5}) has a solution.  
In what follows, $u=u(x,t)$ will denote the unique solution of (\ref{eq:1.1})--(\ref{eq:1.3}).  
Due to \cite[Lemma 6.5]{Ishii_weakKAM_10}, if we set
\[
u_\infty(x)=\liminf_{t\to\infty}u(x,t) \ \ \hb{ for }x\in\ol\gO, 
\]
then $u_\infty\in\Lip(\ol\gO)$ and $u_\infty$ is a solution of (\ref{eq:1.5}). 
Moreover, the proof of \cite[Lemma 6.5]{Ishii_weakKAM_10} shows that the convergence 
\begin{equation}\label{eq:3.2}
u_\infty(x)=\lim_{r\to \infty}\inf\{u(y,t)\mid  t>r\}.
\end{equation}
is uniform for $x\in\ol\gO$.  

Due to \cite[Theorem 7.3]{Ishii_weakKAM_10}, if $\phi$ is a solution of (\ref{eq:1.5}), then for each $x\in\ol\gO$  
there exists a triple $(\eta,\,v,\,l)\in\SP(x)$ such that 
\[
\phi(x)-\phi(\eta(t))=\int_0^t\big(L(\eta(s),\,-v(s))+l(s)g(\eta(s))\big)\d s  
\ \ \hb{ for }t>0. 
\]

According to \cite[Lemma 4.4]{IchiharaIshii_09} or \cite[Lemma 5.2]{DaviniSiconolfi_06}, 
if, in addition to (A1) and (A2), (A5)$_+$ (resp., (A5)$_-$) is satisfied, then 
there is a constant $\gd_1>0$ such that for any $\gd\in[0,\,\gd_1]$ and $(x,\xi)\in S$, 
\begin{equation}\label{eq:3.3}
\begin{gathered}
L(x,\,(1+\gd)\xi)\leq (1+\gd)L(x,\,\xi)+\gd\go_1(\gd),\\
(\hb{resp., } \ L(x,\,(1-\gd)\xi)\leq (1-\gd)L(x,\,\xi)+\gd\go_1(\gd)\,). 
\end{gathered} 
\end{equation}
(The definition of $S$ is given in Section 1 as well as that of $Q$.)

\begin{proof}[Proof of Theorem \ref{thm:1.3}]  
It is enough to show that 
\begin{equation}\label{eq:3.4}
\lim_{t\to \infty}(u(x,t)-u_\infty(x))_+=0 \ \ \hb{ for all }x\in\ol\gO 
\end{equation}
and the convergence is uniform for $x\in\ol\gO$. 
In fact, it is immediate to see from this uniform convergence and (\ref{eq:3.2}) 
that $u(x,t)\to u_\infty(x)$ uniformly 
for $x\in\ol\gO$ as $t\to\infty$. 

Fix any $z\in\ol\gO$. 
Let $(\eta,v,l)\in\SP(z)$ be such that for all $t>0$,
\[
u_\infty(z)-u_\infty(\eta(t))=\int_0^t\big(L(\eta(s),-v(s))+l(s)g(\eta(s))\big)\d s
\]  
Due to Theorem \ref{thm:2.1}, there exists a function $q\in L^\infty(J,\,\R^n)$ such that
\[
\left\{
\begin{aligned}
&\fr{\d}{\d s}u_\infty(\eta(s))=q(s)\cdot \DT \eta(s) \ \ \hb{ for a.e. }s\in J,\\
&H(\eta(s),\,q(s))\leq 0 \ \ \hb{ for a.e. }s\in J, \\ 
&\gg(\eta)\cdot q(s)\leq g(\eta(s)) \ \ \hb{ for a.e. }s\in J_\pl,
\end{aligned}
\right.
\]
where $J_{\pl}:=\{s\in J\mid \eta(s)\in\pl\gO\}$. 

We now show that 
\begin{align}
&H(\eta(s),\,q(s))=0 \ \ \hb{ for a.e. s }\in J, \label{eq:3.5}\\
&\gg(\eta(s))\cdot q(s)=g(\eta(s)) \ \ \ \hb{ for a.e. }s\in J_\pl,\label{eq:3.6}\\
&-q(s)\cdot v(s)=H(\eta(s),\,q(s))+L(\eta(s),\,-v(s)) \ \ \hb{ for  a.e. s }\in J.\label{eq:3.7} 
\end{align}

We remark here that equality (\ref{eq:3.7}) is equivalent to saying that 
\[
-v(s)\in D_2^-H(\eta(s),\,q(s)) \ \ \hb{ for a.e. }s\in J,
\]
or 
\[
q(s)\in D_2^-L(\eta(s),\,-v(s)) \ \ \hb{ for a.e. }s\in J,
\]
where $D_2^-f(x,y)$ stands for the subdifferential with respective to the second 
variable $y$ of the function $f$ on a subset of $\R^n\tim\R^n$ at $(x,y)$.  

Fix any $t>0$.  Noting that 
\[
u_\infty(\eta(t))-u_\infty(\eta(0))=\int_0^t q(s)\cdot \DT \eta(s)\d s, 
\] 
we compute that
\begin{align*}
&\int_0^t\big(L(\eta(s),\,-v(s))+l(s)g(\eta(s))\big)\d s\\
&=u_\infty(z)-u_\infty(\eta(t))=-\int_0^t q(s)\cdot \DT\eta(s)\d s\\
&= \int_0^t q(s)\cdot(l(s)\gg(\eta(s))-v(s))\d s\\
&\leq \int_0^t \big(H(\eta(s),\,q(s))+L(\eta(s),\,-v(s))+l(s)g(\eta(s))\big)\d s\\
&\leq \int_0^t \big(L(\eta(s),\,-v(s))+l(s)g(\eta(s))\big)\d s.
\end{align*}
This series of inequalities ensures that (\ref{eq:3.5})--(\ref{eq:3.7}) hold. 

We fix any $\gep>0$, and prove that there is a constant $\tau>0$ and for each $x\in\ol\gO$ 
a number $\gs(x)\in [0,\,\tau]$ 
for which  
\begin{equation}\label{eq:3.8}
u_\infty(x)+\gep>u(x,\gs(x)).  
\end{equation}

In view of the definition of $u_\infty$, 
for each $x\in\ol\gO$ there is a constant $t(x)>0$ such that 
\[
u_\infty(x)+\gep>u(x,\,t(x)). 
\]
By continuity, for each fixed $x\in\ol\gO$, we can choose a constant $r(x)>0$ 
so that 
\[
u_\infty(y)+\gep>u(y,\,t(x)) \ \ \hb{ for } y\in\ol\gO\cap B_{r(x)}(x),
\] 
where $B_\rho(x):=\{y\in\R^n\mid |y-x|< \rho\}$. 
By compactness, there is a finite sequence $x_i$, $i=1,2,...,N$, such that 
\[
\ol\gO\subset \bigcup_{1\leq i\leq N}B_{r(x_i)}(x_i), 
\]

That is, for any $y\in\ol\gO$ there exists $x_i$, with $1\leq i\leq N$, 
such that $y\in B_{r(x_i)}(x_i)$, which implies 
\[
u_\infty(y)+\gep>u(y,\,t(x_i)). 
\]
Thus, setting 
\[
\tau=\max_{1\leq i\leq N}t(x_i), 
\]
we find that for each $x\in\ol\gO$ there is a constant $\gs(x)\in[0,\tau]$ such that 
(\ref{eq:3.8}) holds.

The rest of the proof is similar to the proof of \cite[Theorem 4.3]
{IchiharaIshii_09}, but for completeness 
we give here the details. 

In what follows we fix $\tau>0$ and $\gs(x)\in[0,\,\tau]$, with $x\in\ol\gO$, 
so that (\ref{eq:3.8}) holds.   
Also, we fix a constant $\gd_1>0$ and a modulus $\go_1$ so that (\ref{eq:3.3}) holds. 

We divide our argument into two cases according to which hypothesis is valid, (A5)$_+$ or (A5)$_-$. 
We first argue under hypothesis (A5)$_+$.  
Choose a constant $T>\tau$ so that $\tau/(T-\tau)\leq\gd_1$. 
Fix any $t\geq T$, and set $\gth=\gs(\eta(t))\in[0,\,\tau]$. 
We set $\gd=\gth/(t-\gth)$ and note that $\gd\leq \tau/(t-\tau)\leq\gd_1$.  
We define 
functions $\eta_\gd$, $v_\gd$, $l_\gd$ on $J$ by 
\begin{align*} 
\eta_\gd(s)=\aln \eta((1+\gd)s), \\
v_\gd(s)=\aln (1+\gd)v((1+\gd)s), \\
l_\gd(s)=\aln (1+\gd)l((1+\gd)s),
\end{align*}
and note that $(\eta_\gd,\,v_\gd,\,l_\gd) \in\SP(z)$. 
By (\ref{eq:3.5}) and (\ref{eq:3.7}), we know that 
$(\eta(s),\,q(s))\in Q$ and 
$(\eta(s),\,-v(s))\in S$ for a.e. $s\in J$. 
Therefore, by (\ref{eq:3.3}), we get  
\[
L(\eta_\gd(s),\,-v_\gd(s))\leq
(1+\gd)L(\eta((1+\gd)s),\,-v(1+\gd)s))+\gd\go_1(\gd) \ \ \hb{ for a.e. } s\in J. 
\]
Integrating this over $(0,\,t-\gth)$, making a change of variables in the integral and 
noting that $(1+\gd)(t-\gth)=t$, we get  
\begin{align*}
\int_0^{t-\gth} L(\eta_\gd(s),\,-v_\gd(s))\d s
\leq \aln \int_0^{t} L(\eta(s),\,-v(s))\d s+(t-\gth)\gd\go_1(\gd)\\  
= \aln \int_0^{t} L(\eta(s),\,-v(s))\d s+\gth\go_1(\gd),  
\end{align*} 
as well as 
\[
\int_0^{t-\gth} l_\gd(s)g(\eta_\gd(s))\d s=\int_0^t l(s)g(\eta(s))\d s.
\]
Moreover, 
\begin{align*}
u(z,t)\leq\aln 
\int_0^{t-\gth} \big(L(\eta_\gd(s),\,v_\gd(s))+l_\gd(s)g(\eta_\gd(s))\big)\d s 
+u(\eta_\gd(t-\gth),\,\gth)\\
\leq\aln \int_0^t \big(L(\eta(s),\,-v(s))+l(s)g(\eta(s))\big)\d s+\gth\go_1(\gd)
+u(\eta_\gd(t-\gth),\,\gth) \\
=\aln u_\infty(z)-u_\infty(\eta(t)) +\tau\go_1(\gd)
+u(\eta(t),\,\gs(\eta(t))) \\
<\aln u_\infty(z)-u_\infty(\eta(t)) +\tau\go_1(\gd) 
+u_\infty(\eta(t))+\gep\\
=\aln u_\infty(z)+\tau\go_1(\gd)+\gep.
\end{align*} 
Thus, recalling that $\gd\leq \tau/(t-\tau)$, we get 
\begin{equation}\label{eq:3.9}
u(z,t)\leq u_\infty(z)+\tau\go_1\Big(\fr \tau{t-\tau}\Big)+\gep.
\end{equation}

Next, we consider the case where (A5)$_-$ is satisfied. 
We choose $T>\tau$ as before, and fix $t\geq T$. Set 
$\gth=\gs(\eta(t-\tau))\in[0,\,\tau]$ and $\gd=(\tau-\gth)/(t-\gth)$. 
Observe that $(1-\gd)(t-\gth)=t-\tau$ 
and $\gd\leq\tau/(t-\tau)\leq \gd_1$. 

We set  
$\eta_\gd(s)=\eta((1-\gd)s)$, $v_\gd(s)= (1-\gd)v((1-\gd)s)$ and 
$l_\gd(s)= (1-\gd)l((1-\gd)s)$ 
for $s\in J$ and observe that $(\eta_\gd,\,v_\gd,\,l_\gd) \in\SP(z)$. 
As before, thanks to (\ref{eq:3.3}), we get  
\[
L(\eta_\gd(s),\,-v_\gd(s))\leq
(1-\gd)L(\eta((1-\gd)s),\,-v(1-\gd)s))+\gd\go_1(\gd) \ \ \hb{ for a.e. }s\in J. 
\]
Hence, we obtain
\begin{align*}
\int_0^{t-\gth} L(\eta_\gd(s),\,-v_\gd(s))\d s
\leq \aln \int_0^{t-\tau} L(\eta(s),\,-v(s))\d s+(t-\gth)\gd\go_1(\gd)\\  
= \aln \int_0^{t-\tau} L(\eta(s),\,-v(s))\d s+(\tau-\gth)\go_1(\gd),  
\end{align*} 
and  
\[
\int_0^{t-\gth} l_\gd(s)g(\eta_\gd(s))\d s=\int_0^{t-\tau} l(s)g(\eta(s))\d s.
\]
Moreover, we get 
\begin{align*}
u(z,t)\leq\aln 
\int_0^{t-\gth} \big(L(\eta_\gd(s),\,v_\gd(s))+l_\gd(s)g(\eta_\gd(s))\big)\d s 
+u(\eta_\gd(t-\gth),\,\gth)\\
\leq\aln \int_0^{t-\tau} \big(L(\eta(s),\,-v(s))+l(s)g(\eta(s))\big)\d s+(\tau-\gth)\go_1(\gd)
+u(\eta(t-\tau),\,\gth) \\
=\aln u_\infty(z)-u_\infty(\eta(t-\tau)) +(\tau-\gth)\go_1(\gd)
+u(\eta(t-\tau),\,\gs(\eta(t-\tau))) \\
<\aln u_\infty(z)-u_\infty(\eta(t-\tau)) +(\tau-\gth)\go_1(\gd) 
+u_\infty(\eta(t-\tau))+\gep\\
=\aln u_\infty(z)+\tau\go_1(\gd)+\gep.
\end{align*} 
Thus, we get
\[
u(z,t)\leq u_\infty(z)+\tau\go_1\Big(\fr{\tau}{t-\tau}\Big)+\gep,
\]

From the above inequality and (\ref{eq:3.9}) we see that if $t\geq T$, then 
\[
u(x,t)\leq u_\infty(x)+\tau\go_1\Big(\fr{\tau}{t-\tau}\Big)+\gep \ \ \hb{ for all }x\in\ol\gO,
\]
which shows that (\ref{eq:3.4}) is valid. The proof is now complete.    
\end{proof}

\section{Remark on asymptotic solutions} 

We continue to assume that the critical value $c$ is zero. Let $u$ denote the unique 
solution of (\ref{eq:1.1})--(\ref{eq:1.3}).   
We set 
\begin{align*}
u_0^-(x)=\aln\sup\left\{\psi(x)\mid \psi\hb{ is a subsolution of }(\ref{eq:1.5}),\ \psi\leq u_0\hb{ on }\ol\gO\right\} \ \ \hb{ for }x\in\ol\gO,\\
u_d^-(x)=\aln\inf\{ d(x,y)+u_0(y)\mid y\in\ol\gO\} \ \ \hb{ for }x\in\ol\gO,\\
u^-(x,t)=\aln\inf\{u(x,t+\tau)\mid \tau>0\} \ \ \hb{ for }(x,t)\in\ol\gO\tim[0,\,\infty),\\
u_0^\infty(x)=\aln\inf\{\phi(x)\mid \phi \hb{ is a solution of }(\ref{eq:1.5}),\ 
\phi\geq u_0^-\hb{ on }\ol\gO\} \ \ \hb{ for }x\in\ol\gO,\\
u_d^\infty(x)=\aln\inf\{d(x,y)+u_d^-(y)\mid y\in\cA\} \ \ \hb{ for }x\in\ol\gO,\\
u_\infty(x)=\aln\sup\{u^-(x,t)\mid t>0\} \ \ \hb{ for }x\in\ol\gO,
\end{align*}
where $d$ denotes the function on $\ol\gO\tim\ol\gO$ given by
\[
d(x,y)=\sup\{\psi(x)-\psi(y)\mid \psi \hb{ is a subsolution of }(\ref{eq:1.5})\},
\]
and $\cA$ denotes the Aubry (or Aubry-Mather) set associated with (\ref{eq:1.5}). 
See \cite{Ishii_weakKAM_10} for the definition of the Aubry set associated with (\ref{eq:1.5}).  

Note here that the proof of Theorem \ref{thm:1.3} shows that 
\[
\lim_{t\to\infty} u(x,t)=u_\infty(x) \ \ \hb{ for }x\in\ol\gO. 
\]
The following proposition gives two other formulas for $u_\infty$, which are valid 
without one of assumptions (A5)$_+$ or (A5)$_-$.

\begin{prop}\label{prop:4.1}
For every $x\in\ol\gO$ we have 
\begin{align}
&u_0^-(x)=u_d^-(x)=u^-(x,\,0), \\
&u_0^\infty(x)=u_d^\infty(x)=u_\infty(x).
\end{align}
\end{prop}

\begin{proof}Let $\psi$ be a subsolution of (\ref{eq:1.5}) satisfying $\psi\leq u_0$ on $\ol\gO$. Then 
we have
\[
\psi(x)\leq \psi(y)+d(x,y)\leq u_0(y)+d(x,y) \ \ \hb{ for }x,y\in\ol\gO. 
\]
Hence, we see that $u_0^-\leq u_d^-$ on $\ol\gO$. 
Next,  since $u_d^-\leq u_0$ on $\ol\gO$, observing by \cite[Theorem 2.7]{Ishii_weakKAM_10} that $u_d^-$ is a subsolution of (\ref{eq:1.5}), 
we find that $u_d^-\leq u_0^-$ on $\ol\gO$. Thus we have $u_0^-=u_d^-$ on $\ol\gO$. 
Now, since $u_0^-\leq u_0$ on $\ol\gO$ 
and $u_0^-(x)$, as a function of $(x,t)$, is a subsolution 
of (\ref{eq:1.1})--(\ref{eq:1.3}), by comparison (see for instance 
\cite[Theorem 3.4]{Ishii_weakKAM_10}) we find that $u_0^-(x)\leq u(x,t)$ 
for $(x,t)\in\ol\gO\tim[0,\,\infty)$. From this it follows that $u_0^-\leq u^-(\cdot,\,0)$ 
on $\ol\gO$.  
On the other hand, $u^-(x,t)$ is a solution of (\ref{eq:1.1}) and (\ref{eq:1.2}) 
and nondecreasing in $t$, and consequently 
$u^-(\cdot,\,0)$ is a subsolution of (\ref{eq:1.5}). Moreover, we have $u^-(\cdot,\,0)\leq u_0$ on 
$\ol\gO$, and conclude that $u^-(\cdot,\,0)\leq u_0^-$ on $\ol\gO$. 
We thus find that $u_d^-=u_0^-=u^-(\cdot,\,0)$ on $\ol\gO$. 

Next let $\phi$ be a solution of (\ref{eq:1.5}) satisfying $\phi\geq u_0^-=u_d^-$ 
on $\ol\gO$. By \cite[Theorem 6.8]{Ishii_weakKAM_10}, we have 
\[
\phi(x)=\inf\{\phi(y)+d(x,y)\mid y\in\cA\} \ \ \hb{ for }x\in\ol\gO,
\]
and hence
\[
\phi(x)\geq \inf\{u_d^-(y)+d(x,y)\mid y\in\cA\} \ \ \hb{ for }x\in\ol\gO.
\]
Accordingly, we get $u_0^\infty\geq u_d^\infty$ on $\ol\gO$. 
Next, if we set $\phi(x)=u_d^-(y)+d(x,y)$ for any fixed $y\in\cA$, 
then $\phi$ is a solution of (\ref{eq:1.5}) and $\phi(x)\geq u_d^-(x)=u_0^-(x)$ 
for $x\in\ol\gO$. Hence we get $u_0^\infty\leq u_d^\infty$ on $\ol\gO$. 
Thus we have $u_0^\infty=u_d^\infty$ on $\ol\gO$.  
Now we note that $u_0^\infty\geq u^-(\cdot,\,0)$ on $\ol\gO$. 
Since $u_0^\infty(x)$ and $u^-(x,t)$, as functions of $(x,t)$, 
are solutions of (\ref{eq:1.1}) and (\ref{eq:1.2}), 
we see by comparison that $u_0^\infty(x)\geq u^-(x,t)$ 
for $(x,t)\in\ol\gO\tim(0,\,\infty)$, from which follows that 
$u_0^\infty\geq u_\infty$ on $\ol\gO$. 
Similarly, we have $u_0^-(x)\leq u^-(x,t)$ for $(x,t)\in\ol\gO\tim[0,\,\infty)$.  
Hence we get $u_0^-\leq u_\infty$ on $\ol\gO$. Since $u_\infty$ is a
solution of (\ref{eq:1.5}), we see that $u_0^\infty\leq u_\infty$ on $\ol\gO$. 
We thus conclude that $u_d^\infty=u_0^\infty=u_\infty$ on $\ol\gO$. 
\end{proof}

The interpretation of the above proposition into the general case of $c$ is straightforward, 
and indeed, 
we just need to replace the solution $u(x,t)$ of (\ref{eq:1.1})--(\ref{eq:1.3}) 
and $H(x,p)$, respectively, by $\tilde u(x,t):=u(x,t)+ct$ 
and $\widetilde H(x,p)=H(x,p)-c$ in the above argument. The conclusion is as follows.   
We set 
\begin{align*}
d(x,y)=\aln \sup\{\psi(x)-\psi(y)\mid \psi \hb{ is a subsolution of }(\ref{eq:1.4})\}
\ \ \hb{ for }x,y\in\ol\gO, \\
u_0^-(x)=\aln\sup\left\{\psi(x)\mid \psi\hb{ is a subsolution of }(\ref{eq:1.4}),\ 
\psi\leq u_0\hb{ on }\ol\gO\right\} \ \ \hb{ for }x\in\ol\gO,\\
u_d^-(x)=\aln\inf\{ d(x,y)+u_0(y)\mid y\in\ol\gO\} \ \ \hb{ for }x\in\ol\gO,\\
u^-(x,t)=\aln\inf\{u(x,t+\tau)+c(t+\tau)\mid \tau>0\} \ \ \hb{ for }(x,t)\in\ol\gO\tim[0,\,\infty),\\
u_0^\infty(x)=\aln\inf\{\phi(x)\mid \phi \hb{ is a solution of }(\ref{eq:1.4}),\ 
\phi\geq u_0^-\hb{ on }\ol\gO\} \ \ \hb{ for }x\in\ol\gO,\\
u_d^\infty(x)=\aln\inf\{d(x,y)+u_d^-(y)\mid y\in\cA\} \ \ \hb{ for }x\in\ol\gO,\\
u_\infty(x)=\aln\sup\{u^-(x,t)\mid t>0\} \ \ \hb{ for }x\in\ol\gO,
\end{align*}
where $\cA$ denotes the Aubry set associated with (\ref{eq:1.4}).  
Then the assertion of Theorem \ref{thm:1.3} together 
with Proposition \ref{prop:4.1} is stated as     
\[
\lim_{t\to\infty}(u(x,t)+ct)=u_\infty(x)=u_d^\infty(x)=
u_0^\infty(x) \ \ \hb{ uniformly on }\ol\gO.
\]

\bibliographystyle{plain} 
\bibliography{ishiiref09}      

\bye